\documentclass{amsart}
\usepackage{amsfonts}

\newtheorem{theorem}{Theorem}[section]
\newtheorem{lemma}[theorem]{Lemma}

\theoremstyle{definition}

\newtheorem{example}[theorem]{Example}

\theoremstyle{remark}
\newtheorem{remark}[theorem]{Remark}

\numberwithin{equation}{section}

\begin{document}

\title{On the Exponential  Stability of Switching-Diffusion Processes
with Jumps}

\author{Chenggui Yuan}

\author{Jianhai Bao}
\address{Department of Mathematics,
 Swansea University, Swansea SA2 8PP, UK}
\email{majb@swansea.ac.uk}


\date{July 6, 2011 and, in revised form, August 3, 2011.}



\begin{abstract}
In this paper we focus on the pathwise stability of mild solutions
for a class of
 stochastic partial differential equations which are driven by switching-diffusion processes with jumps. In comparison to the existing literature, we show that:
 (i) the criterion
 to guarantee pathwise stability does not rely on the moment stability of the system; (ii) the sample
Lyapunov exponent obtained is generally smaller than that of the
counterpart driven by a Wiener process; (iii) due to the Markovian
switching the overall system can become  pathwise exponentially
stable although some subsystems are not stable.
\end{abstract}

\maketitle


.

\section{Introduction}
Stochastic partial differential equations (SPDEs) have been widely
used to model  phenomena arising in many branches of science such as
ecology, economics, mechanics, biology and chemistry, e.g.,
Applebaum \cite{a09}, Chow \cite{chow07}, Da Prato and Zabczyk
\cite{dz92}, Liu \cite{l04}, Peszat and Zabczyk \cite{pz07}, and
Woyczy\'{n}ski \cite{w01}. Recently,  hybrid systems, in which
continuous dynamics are intertwined with discrete events, have also
been used to model many such systems. One of the distinct features
of hybrid systems is that the underlying dynamics are subject to
changes with respect to  certain configurations. For example,
consider a one-dimensional rod of length $\pi$ whose ends are
maintained at $0^\circ$ and whose sides are insulated. Assume that
there is an exothermic reaction taking place inside the rod with
heat being produced proportionally to the temperature. The
temperature $u$ in the rod may be modelled by (see, e.g.,
\cite{H78})
\begin{equation}\label{eq111}
\begin{cases}
\frac{\partial u}{\partial t}=\frac{\partial^2u}{\partial x^2}+cu,
\ \ \ \ \ \ \ \ t>0, \ \ x\in(0,\pi),\\
u(t,0)=u(t,\pi)=0,\\
u(0,x)=u_0(x),
\end{cases}
\end{equation}
where  $u=u(t,x)$ and $c$ is a constant dependent on the rate of
reaction. The system \eqref{eq111} will switch from one mode to
another in a random way when it experiences abrupt changes in its
structure and parameters caused by phenomena such as component
failures or repairs, changing subsystem interconnections, or abrupt
environmental disturbances. A hybrid system driven by a
continuous-time Markov chain can be applied to describe such a
situation. The system \eqref{eq111} under regime switching could  be
described by the following stochastic model
\begin{equation*}
\begin{cases}
\frac{\partial u}{\partial t}=\frac{\partial^2u}{\partial
x^2}+c(r(t))u,
\ \ \ \ \ \ \ \ t>0, \ \ x\in(0,\pi),\\
u(t,0)=u(t,\pi)=0,\\
u(0,x)=u_0(x), \ \ \ r(0)=r_0,
\end{cases}
\end{equation*}
where $r(t)$ is a right-continuous Markov chain with finite state
space $\mathbb{S}$ and $c:\mathbb{S}\rightarrow\mathbb{R}$. As a
second example, Li et al. \cite{llp09} discussed stochastic
age-dependent population equation with Markovian switching,
\begin{equation*}
\begin{cases}
\frac{\partial P}{\partial t}+\frac{\partial P}{\partial
a}=-\mu(r(t),
a)P +f(r(t),P)+g(r(t), P)\frac{dW(t)}{dt}, \ \ (t,a)\in Q,\\
\ \ P(0,a)=P_0(a), \ \ \ r(0)=r_0,\ \ \ \ \ \ \ \ \ \ \ \ \ \ \ \  \ \ \ \ \ \ \ \ \ \  \ \ \ \ \ \ \ \ a\in[0,A],\\
\ \ \ P(t,0)=\int_0^A\beta(t,a)P(t,a)da, \ \ \ \ \ \ \ \ \ \ \ \ \ \
\ \  \ \ \ \ \ \ \ \ \ \  \ \ \ \ \ \ \ t\in[0,T],
\end{cases}
\end{equation*}
where $T > 0, A > 0, Q:=(0, T) \times(0, A), P=P(t,a)$ is the
population density of age $a$ at time $t$, $r(t)$ is a
right-continuous Markov chain, $\mu(r(t), a)$ denotes the mortality
rate of age $a$ at time $t$, and $f(r(t),P)$ denotes the effects of
the external environment on the population system. For more
quantitative analysis of
 SPDEs  with Markovian regime switching, we refer to
Anabtawi and Ladde \cite{al00}, Anabtawi  and Sathananthan
\cite{as08}, Luo and Liu \cite{ll08}, and the references therein.
For the finite-dimensional case, we  refer to the monographs of Mao
and Yuan \cite{my06} and Yin and Zhu \cite{YZ}.

Non-Gaussian random processes also play an important role in
modelling stochastic dynamical systems, e.g., Applebaum \cite{a09},
{\O}ksendal  and Sulem \cite{os07}, and
 Peszat and Zabczyk \cite{pz07}. Typical examples of non-Gaussian stochastic processes are L\'{e}vy processes
and processes arising from Poisson random measures. The monograph
\cite{w01} describes a number of phenomena from fluid mechanics,
solid state physics, polymer chemistry, economic science, etc.,
which can be modelled  using non- Gaussian L\'{e}vy processes.

Moreover, one of the most important and interesting problems in the
analysis of SPDEs is their stability. Stability issues for SPDEs
with Wiener noise  are by now classical, see, e.g., \cite{clm01,
chow07, ich82, l04,ll08}, but comparable theories driven by jump
noise are not yet fully developed. Recent years have witnessed a
growing interest in this area: In  \cite{bty09} we investigated the
asymptotic stability in distribution of mild solutions for delay
equations using Lyapunov functions and the Yorsida approximation; by
energy inequality approach, moment stability and sample path
stability for variational solutions were discussed in \cite{hby10},
and almost sure exponential stability was studied in \cite{ll08}
provided that the mild solution is  moment exponentially stable. For
the finite-dimensional case, Applebaum and Siakalli \cite{as09}
provide sufficient conditions under which the solutions to
stochastic differential equations (SDEs) driven by L\'{e}vy noise
are stable in probability, almost surely and moment exponentially
stable. As a sequel, Applebaum and Siakalli \cite{as10} made some
first steps in the stochastic stabilization problems where the noise
source is a L\'{e}vy noise, i.e., a Brownian motion and an
independent Poisson random measure.

Most of the previous literature does not discuss the sample Lyapunov
exponents which are given explicitly by the parameters arising from
the jump-diffusion coefficients or stationary probability
distribution of  Markovian chain. In this paper, motivated by the
previous references,  we shall study the pathwise stability of mild
solutions for a class of
 SPDEs of the form
\begin{equation}\label{eq17}
\begin{split}
dX(t)&=[AX(t)+F(t,X(t),r(t))]dt+G(t,X(t),r(t))dW(t)\\
&\quad+\int_{\mathbb{Z}}\Phi(t,X(t^-),r(t),u)\tilde{N}(dt,du)
\end{split}
\end{equation}
on $t \geq 0$ with the initial data $X(0)=x_0\in H$ and $r(0)=r_0\in
\mathbb{S}$, where $X(t^-):=\lim_{s\uparrow t}X(s)$. More detailed
information on the parameters in Eq. \eqref{eq17} will be given in
Section $2$.

In comparison to the existing literature for the almost sure
exponential stability of solutions to SPDEs, our main result
(Theorem 3.1), has the following advantages: (i) the criterion
established does not rely on the moment exponential stability of the
system; (ii) the sample Lyapunov exponent is generally smaller than
that of the counterpart driven by a Wiener process; (iii) due to the
Markovian switching the overall system could become pathwise
exponentially stable, although some subsystems are not stable.

Our approach is  based on the Yosida approximation and a classical
Lyapunov function argument. Since mild solutions do not necessarily
have stochastic differentials, one cannot apply the It\^o formula
directly. To overcome this problem, we first apply the It\^o formula
to an approximating equation, and then investigate   the stability
properties of the mild solutions. This approach is  dependent on a
maximal inequality (Burkh\"older-Davis-Gundy inequality) for
stochastic convolutions with jumps.

\section{Preliminaries}
Let $(H,\langle\cdot,\cdot\rangle_H, \|\cdot\|_H)$ be a Hilbert
space and  $W(t)$ a cylindrical Wiener process on $H$ defined on
some filtered probability space $(\Omega, \mathcal {F}, \mathbb{P})$
equipped with a filtration ${\{\mathcal {F}_t}\}_{t\geq0}$
satisfying the usual conditions. Let $\mathbb{Z}$ be a vector space
endowed with a norm $|\cdot|$, $\mathcal {B}(\mathbb{Z})$ is the
Borel $\sigma$-algebra on $\mathbb{Z}$, and $\lambda(dx)$ a
$\sigma$-finite measure defined on $\mathcal {B}(\mathbb{Z})$. Let
$p=(p(t)),t\in D_p$, be a stationary $\mathcal {F}_t$-Poisson point
process on $\mathbb{Z}$ with characteristic measure $\lambda$.
Denote by $N(dt,du)$ the Poisson counting measure associated with
$p$, i.e., $N(t,\mathbb{Y})=\sum_{s\in D_p, s\leq
t}I_{\mathbb{Y}}(p(s))$ for $\mathbb{Y}\in\mathcal {B}(\mathbb{Z})$.
Let $\tilde{N}(dt,du):=N(dt,du)-dt\lambda(du)$ be the compensated
Poisson measure. Let $m$ be some positive integer and $\{r(t),t\in
\mathbb{R}_{+}\}$
  a right continuous irreducible Markov chain on
the probability space $\{\Omega,{\mathcal F},  \mathbb{P}\}$ taking
values in a finite state space $\mathbb{S}:=\{1,2,...,m\}$, with
generator $\Gamma=(\gamma_{ij})_{m\times m}$ given by
$$ \mathbb{P}(r(t + \Delta )=j|r(t)=i) = \left\{
\begin{array}{cc}
 \gamma_{ij} \Delta  + o(\Delta ),\ \ & \ \ {\rm  if}\ \ i \ne j, \\
 1 + \gamma_{ii} \Delta  + o(\Delta ),\ \ & \ \ {\rm  if}\ \ i = j, \\
 \end{array} \right.$$
where $\Delta>0$ and $\gamma_{ij}\geq 0$ is
 the transition rate from $i$ to $j$, if $i\ne j$; while $\gamma_{ii}=-\sum_{j\ne
 i}\gamma_{ij}$. For a mapping $\zeta:\mathbb{S}\rightarrow\mathbb{R}$, we
 write
$\zeta_i:=\zeta(i)$. We further assume that the Wiener process
$W(t),$ Poisson process $N(t,\cdot)$ and Markov chain $r( t)$ are
independent. For $T>0$ and $p>0$, let $D([0,T];H)$ be the family of
all c\`{a}dl\`{a}g paths from $[0,T]$ into $H$ with the uniform norm
and $L^p:=L^p(\Omega, \mathcal {F}, \mathbb{P};D([0,T];H)):=\{X:
\mathbb{E}\sup_{0\leq t\leq T}\|X(t)\|_H^p <\infty\}$.

\begin{remark}
{\rm  Since we have assume that the Markov chain $r(t)$ is
irreducible, it has a unique stationary probability distribution
$\pi:=(\pi_1,\pi_2,\ldots,\pi_m)\in\mathbb{R}^{1\times m}$ which can
be determined by solving
\begin{equation*}
\pi\Gamma=0 \ \ \ \mbox{ subject to } \ \ \ \sum_{j=1}^m\pi_j=1
\mbox{ and } \pi_j>0,\ \ \  \forall j\in\mathbb{S}.
\end{equation*}
}
\end{remark}

 Throughout the
paper we will assume:
\begin{enumerate}
\item[\textmd{({\it H1})}] $A:\mathcal
{D}(A)\rightarrow H$ is the infinitesimal generator of a
pseudo-contraction $C_0$-semigroup $S(t)$ on $H$, that is,
\begin{equation}\label{eq22}
\|S(t)\|\leq e^{\alpha t}, \ \ \ \ \ \ \mbox{ for some }\alpha\geq0.
\end{equation}
\end{enumerate}

\begin{enumerate}
\item[\textmd{({\it H2})}]
$F: [0,T]\times H\times\mathbb{S}\rightarrow H$, $G:[0,T]\times
H\times\mathbb{S}\rightarrow \mathcal {L}_2(H,H)$, the family of
Hilbert-Schmidt operators from $H$ into itself,  and
$\Phi:[0,T]\times H\times{\mathbb{Z}}\times\mathbb{S} \rightarrow
H$, satisfy Lipschitz and linear growth conditions, i.e., there
exist positive constants $L, \tilde L$ such that for arbitrary
$x,y\in H, i\in\mathbb{S}$
\begin{equation*}
\begin{split}
\| F(t,x,i)&-F(t,y,i)\|_H^2+\| G(t,x,i)-G(t,y,i)\|_{HS}^2\leq L\| x-y\|_H^2,\\
& \| F(t,x,i) \|_H^2+\| G(t,x,i)\|_{HS}^2\leq
\tilde{L}(1+\| x\|_H^2),\\
\end{split}
\end{equation*}
\begin{equation*}
\begin{split}
&\int_{\mathbb{Z}}\| \Phi(t,x,i,
u)-\Phi(t,y,i,u)\|_H^2\lambda(du) \leq L\| x-y\|_H^2,\\
&\int_{\mathbb{Z}}\| \Phi(t,x,i, u)\|_H^2\lambda(du)\leq
\tilde{L}(1+\| x\|_H^2).
\end{split}
\end{equation*}
\end{enumerate}
Moreover, we will need some slightly  stronger conditions on $\Phi$:
\begin{enumerate}
\item[\textmd{({\it H3})}] For $p\geq2$, there exists $\bar{L}>0$ such that
\begin{equation}\label{eq25}
\int_{\mathbb{Z}}\| \Phi(t,x,i,
u)-\Phi(t,y,i,u)\|_H^p\lambda(du)\leq \bar{L}\|x-y\|_H^p,\ \ \ \
x,y\in H, i\in\mathbb{S},
\end{equation}
and  $\Phi(t,0,i, u)=0$.
\end{enumerate}
\begin{enumerate}
\item[\textmd{({\it H$3^{\prime}$})}]
There exists an $L>0$ such that for arbitrary $x,y\in H,
i\in\mathbb{S}, u \in \mathbb{Z}$ and $p\geq2$
\begin{equation*}
\| \Phi(t,x,i, u)-\Phi(t,y,i,u)\|_H \leq L\| x-y\|_H|u| \mbox{ and }
\int_{\mathbb{Z}}|u|^p\lambda(du)<\infty.
\end{equation*}
\end{enumerate}
Clearly $(H3^{\prime})$ implies $(H3)$ and $\| \Phi(t,x,i, u)\|_H
\leq L\| x\|_H|u|$ whenever $\Phi(t,0,i, u)=0$.

In this paper we are mainly concerned with  mild solutions to
\eqref{eq17}. For the notion of mild solutions, we can refer to,
e.g., \cite[Definition 2.1]{mpr10}, for SPDEs driven by
multiplicative Poisson noise, and, with an obvious extension, we can
define such solutions for the case of switching-diffusion processes
with jumps.

\begin{lemma}\label{global existence}
{\rm Under conditions $(H1)$ and $(H2)$, Eq. \eqref{eq17} admits a
unique mild solution $X(t,x_0,r_0), t\in[0,T]$, in $L^2(\Omega,
\mathcal {F}, \mathbb{P};H)$.}
\end{lemma}

\begin{proof}
 Since the proof of Lemma \ref{global existence} can be done by combining classical Banach fixed-point theorem argument
\cite[Theorem 9.29, p164]{pz07} with stopping time techniques
\cite[Theorem 3.13, p89]{my06}, we will only give a sketch of the
argument. Recall that almost every sample path of $r(\cdot)$ is a
right-continuous step function with a finite number of sample jumps
on $[0,T]$. So for almost every $\omega\in\Omega$ there is a finite
$k=k(\omega)$ such that $0=\tau_0<\tau_1<\cdots<\tau_{k}\geq T$  and
\begin{equation*}
r(t)=r(\tau_k) \mbox{ on } \tau_k\leq t<\tau_{k+1} \mbox{ for }
\forall k\geq0.
\end{equation*}
We first consider Eq. \eqref{eq17} on $t\in[0,\tau_1]$ which becomes
\begin{equation}\label{eq007}
\begin{split}
dX(t)&=[AX(t)+F(t,X(t),r_0)]dt+G(t,X(t),r_0)dW(t)\\
&\quad+\int_{\mathbb{Z}}\Phi(t,X(t^-),r_0,u)\tilde{N}(dt,du)
\end{split}
\end{equation}
with the initial data $X(0)=x_0\in H$ and $r(0)=r_0\in \mathbb{S}$.
By \cite[Theorem 9.29, p164]{pz07} Eq. \eqref{eq007} admits a unique
mild solution $X(t,x_0,r_0)$ for  $t\in[0,\tau_1]$ which belongs to
$L^2(\Omega, \mathcal {F}, \mathbb{P};H)$. Next consider Eq.
\eqref{eq1} on $t\in[\tau_1,\tau_2]$ which becomes
\begin{equation}\label{eq008}
\begin{split}
dX(t)&=[AX(t)+F(t,X(t),r(\tau_1))]dt+G(t,X(t),r(\tau_1))dW(t)\\
&\quad+\int_{\mathbb{Z}}\Phi(t,X(t^-),r(\tau_1),u)\tilde{N}(dt,du)
\end{split}
\end{equation}
with the initial data $X(\tau_1,x_0,r_0)$ and $r(\tau_1)$. Again by
\cite[Theorem 9.29, p164]{pz07} Eq. \eqref{eq008} admits a unique
mild solution $X(t,X(\tau_1,x_0,r_0),r(\tau_1))$ for
$t\in[\tau_1,\tau_2]$ which belongs to $L^2(\Omega, \mathcal {F},
\mathbb{P};H).$ The proof can then  be completed by repeating this
procedure.
\end{proof}

Let $U: \mathbb{R}_+\times H\times
\mathbb{S}\rightarrow\mathbb{R}_+$ be a $C^{1,2}$-function. For
$t\geq0, x\in \mathcal {D}(A)$ and $i\in\mathbb{S}$, define an
operator
\begin{equation*}
\begin{split}
\mathcal {L}U(t,x,i)&:=U_t(t,x,i)+\langle Ax+F(t,x,i),U_x(t,x,i)\rangle_H\\
&\quad+\sum_{j=1}^m\gamma_{ij}U(t,x,j)+\frac{1}{2}\mbox{trace}(U_{xx}(t,x,i)G(t,x,i)G^*(t,x,i))\\
&\quad+\int_{\mathbb{Z}}[U(t,x+\Phi(t,x,i,u),i)-U(t,x,i)-\langle
U_x(t,x,i), \Phi(t,x,i,u)\rangle_H]\lambda(du).
\end{split}
\end{equation*}
Similarly to \cite[Theorem 1.45, p48]{my06} and \cite [Theorem D.2,
p392] {pz07}, for  a strong solution $X(t)$ of Eq. \eqref{eq17}, we
have the following It\^o formula:
\begin{equation}\label{ito}
\begin{split}
&U(t,X(t),r(t))\\
&=U(0,x_0,r_0)+\int_0^t\mathcal {L}U(s,X(s),r(s))ds \\
&\quad+\int_0^t\langle U_x(s,X(s),r(s)),G(s,X(s),r(s))dW(s)\rangle_H \\
&\quad+\int_0^t\int_{\mathbb{Z}}[U(s,X(s^{-})+\Phi(s,X(s^{-}),r(s),u),r(s))-U(s,X(s^{-}),r(s))]\tilde{N}(ds,du)\\
&\quad+\int_0^t\int_\mathbb{R}[U(s,X(s^{-}),
r_0+h(r(s),\ell))-U(s,X(s^{-}),r(s))\mu(ds,d\ell),
\end{split}
\end{equation}
where  $\mu (ds,d\ell)$ is a Poisson random measure with intensity
$ds \times m(d\ell)$, in which $m$ is the Lebesgue measure on
$\mathbb{R}$. For more details on the function $h$ and the
martingale measure $\mu(ds, d\ell)$,  see, e.g.,
\cite[p46-48]{my06}.

Since mild solutions do not necessarily have stochastic
differentials, one cannot apply the It\^o formula directly. Instead,
we introduce the Yosida-approximation system
\begin{equation}\label{eq24}
\begin{split}
dX^l(t)=&[AX^l(t)+R(l)F(t,X^l(t),r(t))]dt+R(l)G(t,X^l(t),r(t))dW(t)\\
&+\int_{\mathbb{Z}}R(l)\Phi(t,X^l(t^-),r(t),u)\tilde{N}(dt,du)
\end{split}
\end{equation}
with the initial data $X^l(0)=R(l)x_0$ and $r(0)=r_0$. Here
$l\in\rho(A)$, the resolvent set of $A$, and $R(l):=l
R(l,A):=l(lId-A)^{-1}$, where
  $Id$ is the identity operator from $H$ into itself.

\begin{theorem}\label{approximation}
{\rm Under conditions $(H1)-(H3)$, Eq. \eqref{eq24} has a unique
strong solution $X^l(t)$ in $L^p,p\geq2$. Moreover, $X^l(t)$
converges to the mild solution $X(t)$ of Eq. \eqref{eq17} in $L^p$
as $l\rightarrow\infty$, i.e.,
\begin{equation*}
\lim\limits_{l\rightarrow\infty}\mathbb{E}\left(\sup\limits_{t\in[0,T]}\|X^l(t)-X(t)\|_H^p\right)=0.
\end{equation*}
In particular, there exists a subsequence, still denoted by
$X^l(t)$, such that $X^l(t)\rightarrow X(t)$ almost surely as
$l\rightarrow\infty$ uniformly in $[0, T]$. }
\end{theorem}

\begin{proof}
This can be proven by following the argument in \cite[Proposition
2.4]{ll08} and using the result of Lemma \ref{B-D-G} below.
\end{proof}

\begin{lemma}\label{B-D-G}
{\rm ( \cite[Theorem 4.4]{BHZ} or \cite[Proposition 3.3]{mpr10})
Assume that $\Phi:\Omega\times
\mathbb{R}_+\times\mathbb{Z}\rightarrow H$, is a progressively
measurable process, and for  $p\geq2$
\begin{equation}\label{10}
\mathbb{E}\int_0^T\int_{\mathbb{Z}}\|\Phi(s,u)\|_H^p\lambda(du)ds<\infty.
\end{equation}
If $S(t),t\in[0,T]$, is a pseudo-contraction $C_0$-semigroup such
that \eqref{eq22}, then
\begin{equation}\label{eq33}
\mathbb{E}\left[\sup\limits_{0\leq t\leq
T}\left\|\int_0^t\int_{\mathbb{Z}}S(t-s)\Phi(s,u)\tilde{N}(ds,du)\right\|_H^p\right]\leq
C_p\mathbb{E}\int_0^T\int_{\mathbb{Z}}\|\Phi(s,u)\|_H^p\lambda(du)ds,
\end{equation}
where $C_p$ is a positive constant dependent on $p,\alpha,T$. }
\end{lemma}

\section{A Criterion for Sample Path Stability}
In this section we give a criterion for sample path stability of
mild solutions to Eq. \eqref{eq17}.
\begin{theorem}\label{asymptotic}
Let  $(H1), (H2)$ and ($H3^{\prime}$) hold. Assume that the solution
of Eq. \eqref{eq17} is such that $X(t)\neq0$ a.s. for all $t\geq0$
and $i\in\mathbb{S}$ provided $x_0\neq0$ a.s. For $U\in
C^{1,2}(\mathbb{R}_+\times H\times \mathbb{S};\mathbb{R}_+)$, assume
further that  there exist constants $c_2>c_1>0, c_3>0,
p>0,\alpha_i,\rho_i\in\mathbb{R}, \beta_i,\delta_i\geq0$ such that for $(t,x,i)\in\mathbb{R}_+\times \mathcal {D}(A)\times\mathbb{S}$\\
\begin{enumerate}
\item[\textmd{(i)}] $c_1\|x\|_H^p\leq U(t,x,i)\le c_2\|x\|_H^p,$ \  $\|U_x(t, x, i)\|_H\|x\|_H + \|U_{xx}(t, x, i)\|\|x\|_H^2\le c_3\|x\|_H^p;$\\
\item[\textmd{(ii)}] $\mathcal {L}U(t,x,i)\leq\alpha_iU(t,x,i)$;\\
\item[\textmd{(iii)}] $\Theta U(t,x,i):=
\|G^*(t,x,i)U_x(t,x,i)\|_H^2
\geq\beta_iU^2(t,x,i)$;\\
\item[\textmd{(iv)}] For $\Psi(t,x,i,j):=U(t,x,j)/U(t,x,i)$,
\begin{equation*}
\sum\limits_{j=1}^m\gamma_{ij}(\ln\Psi(t,x,i,j)-\Psi(t,x,i,j))\leq\rho_i;
\end{equation*}
\item[\textmd{(v)}] For
$\Lambda(t,x,i,u):=U(t,x+\Phi(t,x,i,u),i)/U(t,x,i)$,
\begin{equation*} \int_{\mathbb{Z}}\left[\ln
\Lambda(t,x,i,u)-\Lambda(t,x,i,u)+1\right]\lambda(du):=J(t,x,i)\leq-\delta_i,
\end{equation*}
and for some $\epsilon\in(0,\frac{1}{2}]$
\begin{equation*}
\zeta:=\limsup_{t\rightarrow\infty}\frac{1}{t}\Xi(t)<\infty,
\end{equation*}
where
\begin{equation*}
 \Xi(t):=\int_0^t\int_{\mathbb{Z}}[(\ln
\Lambda(s,x,i,u))^2+\Lambda^\epsilon(s,x,i,u)]\lambda(du)ds;
\end{equation*}
\item[\textmd{(vi)}] For $\Upsilon(t, x, r_0,
i,\ell):=U(t,x,r_0+h(i,\ell))/U(t,x, i)$ and some
$\tilde{\epsilon}\in(0,\frac{1}{2}]$
\begin{equation*}
\eta:=\limsup\limits_{t\rightarrow\infty}\frac{1}{t}\Pi(t)<\infty,
\end{equation*}
where
\begin{equation*}
\Pi(t):=\int_0^t\int_{\mathbb{Z}}[(\ln \Upsilon(s, x, r_0,
i,\ell))^2+\Upsilon^{\tilde{\epsilon}}(s, x, r_0,
i,\ell)]m(d\ell)ds.
\end{equation*}
\end{enumerate}
Then the mild solution of Eq. \eqref{eq17} has the property
\begin{equation*}
\limsup\limits_{t\rightarrow\infty}\frac{1}{t}\ln(\|X(t)\|_H)\leq-\dfrac{1}{p}\sum\limits_{i=1}^m\pi_i\left(\frac{1}{2}\beta_i+\delta_i-\alpha_i-\rho_i\right),\
\ \ \ \mbox{a.s.}
\end{equation*}
In particular, the mild solution of Eq. \eqref{eq17} is almost
surely exponentially stable provided that
\begin{equation*}
\sum\limits_{i=1}^m\pi_i\left(\frac{1}{2}\beta_i+\delta_i-\alpha_i-\rho_i\right)>0.
\end{equation*}
\end{theorem}

\begin{remark}
 By the fundamental inequality
\begin{equation*}
\ln(1+x)\leq x \ \ \mbox{ for } x\geq0,
\end{equation*}
it is easy to observe that the first assumption in (v) is
reasonable. On the other hand, for $x\in H$ and $i,j\in\mathbb{S}$,
 (vi) is also true provided that there exist constants $c_4, c_5>0$
such that
\begin{equation*}
c_4\leq U(t,x,j)/U(t,x,i)\leq c_5.
\end{equation*}
There are many functions possessing this property, e.g., for $x\in
H, i\in\mathbb{S}$,  $U(t,x,i)=\sigma_i\|x\|_H^2$ with $\sigma_i>0$.
For this Lyapunov function, we have
$\Psi(t,x,i,j)=\sigma_j/\sigma_i$ and then  (iv) also holds.
\end{remark}

\noindent{\bf Proof of Theorem \ref{asymptotic}.} As we mentioned
previously, one cannot apply the It\^o formula to the mild solutions
directly.  we first apply the It\^o formula to the approximation
equation \eqref{eq24}, then use   Theorem \ref{approximation} to
investigate  the stability of the mild solutions.  More precisely,
applying the It\^o formula \eqref{ito} to $\ln U(t,x,i)$ with
respect to $X^l(t),t\geq0$, where $X^l(t)$ denotes the strong
solution of Eq. \eqref{eq24}, we obtain
\begin{equation*}
\begin{split}
\ln U(t,X^l(t),r(t)) &=\ln U(0,R(l)x_0,r_0)+\int_0^t\frac{\mathcal
{L}U(s,X^l(s),r(s))}{U(s,X^l(s),r(s))}ds\\
&\quad+\int_0^t\sum_{j=1}^m\gamma_{r(s)j}\left[\ln \Psi(s,X^l(s),i,j)-\Psi(s,X^l(s),i,j)\right]ds\\
&\quad+\int_0^t\int_{\mathbb{Z}}\ln\Big(
\frac{U(s,X^l(s^-)+R(l)\Phi(s,X^l(s^-),r(s),u),r(s))}{U(s,X^l(s^-),r(s))}\Big)\tilde{N}(ds,du)\\
&\quad+\int_0^t\int_\mathbb{R}\ln \Upsilon(t, X^l(s^-), r_0,
r(s),\ell)\mu(ds,d\ell)\\
 &\quad+J_1(t, l)+J_2(t,l)+J_3(t,l),
\end{split}
\end{equation*}
where
\begin{equation*}
\begin{split}
J_1(t, l)&: =\frac{1}{2}\int_0^t \mbox{trace} \bigg[\bigg(\frac{U_{xx}(s,X^l(s),r(s))}{U(s,X^l(s),r(s))}-\frac{U_{x}(s,X^l(s),r(s))\otimes U_{x}(s,X^l(s),r(s))}{U^2(s,X^l(s),r(s))}\bigg)\\
&\ \ \ \ \ \ \ \ \ \ \  \ \ \ \ \ \ \ \ \ \ R(l)G(s,X^l(s),r(s))(R(l)G(s,X^l(s),r(s)))^*\bigg]ds\\
&\quad-\frac{1}{2}\int_0^t \frac{\mbox{trace}\Big[U_{xx}(s,X^l(s),r(s))G(s,X^l(s),r(s))(G(s,X^l(s),r(s)))^*\Big]}{U(s,X^l(s),r(s))}ds\\
J_2(t,l)&:=\int_0^t\frac{\langle U_x(s,X^l(s),r(s)), (R(l)-I)F(s,X^l(s),r(s))\rangle_H}{U(s,X^l(s),r(s))}ds,\\
J_3(t,l)&:=\int_0^t\int_{\mathbb{Z}}\Big[\ln\Big(
\frac{U(s,X^l(s)+R(l)\Phi(s,X^l(s),r(s),u),r(s))}{U(s,X^l(s),r(s))}\Big)\\
&\quad\ \ \ \ \ \ \ \ \ \ \   -\frac{\langle
U_x(s,X^l(s),r(s)),(R(l)-I)\Phi(s,X^l(s),r(s),u),r(s))}{U(s,X^l(s),r(s))}\\
&\quad\ \ \ \ \ \ \ \ \ \ \  -\frac{U(s,X^l(s)+\Phi(s,X^l(s),r(s),u),r(s))}{U(s,X^l(s),r(s))}+1\Big]\lambda(du)ds.\\
\end{split}
\end{equation*}
By Theorem \ref{approximation}, ($H3^{\prime}$), (i) and the
Dominated Convergence Theorem, we have almost surely
\begin{equation*}
\begin{split}
&\lim_{l \rightarrow \infty}J_1(t, l)=-\frac{1}{2}\int_0^t \frac{\Theta U(s,X(s),r(s))}{U(s,X(s),r(s))}d s, \ \ \ \  \lim_{l \rightarrow \infty}J_2(t, l)=0,\\
&\lim_{l \rightarrow \infty}J_3(t, l)=\int_0^t J(s,X(s),r(s))ds.
\end{split}
\end{equation*}
If we let $l \rightarrow \infty,$ then,
\begin{equation*}
\begin{split}
\ln U(t,X(t),r(t)) &=\ln
U(0,x_0,r_0)+\int_0^tJ(s,X(s),r(s))ds\\
&\quad+\int_0^t\frac{\mathcal
{L}U(s,X(s),r(s))}{U(s,X(s),r(s))}ds-\frac{1}{2}\int_0^t\frac{\Theta U(s,X(s),r(s))}{U^2(s,X(s),r(s))}ds\\
&\quad+\int_0^t\frac{\langle
U_x(s,X(s),r(s)),g(s,X(s),r(s))dW(s)\rangle_H}{U(s,X(s),r(s))}\\
&\quad+\int_0^t\sum_{j=1}^m\gamma_{r(s)j}\left[\ln \Psi(s,X(s),i,j)-\Psi(s,X(s),i,j)\right]ds\\
&\quad+\int_0^t\int_{\mathbb{Z}}\ln\Lambda(s,X(s^-),r(s),u)\tilde{N}(ds,du)\\
&\quad+\int_0^t\int_\mathbb{R}\ln \Upsilon(s, X(s^-), r_0,
r(s),\ell)\mu(ds,d\ell).
\end{split}
\end{equation*}
By the exponential martingale inequality with jumps \cite[Theorem
5.2.9, p291]{a09}, for any $T, \theta,\nu>0$
\begin{equation*}
\begin{split}
\mathbb{P}\Big\{&\omega:\sup\limits_{0\leq t\leq
T}\Big[\int_0^t\frac{\langle
U_x(s,X(s),r(s)),g(s,X(s),r(s))dW(s)\rangle_H}{U(s,X(s),r(s))}-\frac{\theta^2}{2}\int_0^t\frac{\Theta U(s,X(s),r(s))}{U^2(s,X(s),r(s))}ds\\
&+\int_0^t\int_{\mathbb{Z}}\ln\Lambda(s,X(s^-),r(s),u)\tilde{N}(ds,du)\\
&-\frac{1}{\theta}\int_0^t\int_\mathbb{Z}\Big[\Lambda^\theta(s,X(s),r(s),u)-1-\theta
\ln \Lambda(s,X(s),r(s),u)\Big]\lambda(du)ds\Big]>\nu\Big\}\\
&\leq e^{-\theta\nu},
\end{split}
\end{equation*}
and, if we denote $\tilde{\Upsilon}(t,\ell):=\Upsilon(t, x, r_0,
i,\ell)$, then
\begin{equation*}
\begin{split}
\mathbb{P}\Big\{\omega:\sup\limits_{0\leq t\leq
T}\Big[&\int_0^t\int_\mathbb{R}\ln \tilde{\Upsilon}(s, \ell)\mu(ds,d\ell)\\
&-\frac{1}{\theta}\int_0^t\int_{\mathbb{R}}\Big[\tilde{\Upsilon}^\theta(s,\ell)-1-\theta\ln
\tilde{\Upsilon}(s, \ell)\Big]m(d\ell)ds\Big]>\nu\Big\}\leq
e^{-\theta\nu}.
\end{split}
\end{equation*}
Taking $ T=n, \  \nu=2\theta^{-1}\ln n, \ n=1,2,\ldots, $ for
$\theta\in(0,\frac{\epsilon\wedge\tilde{\epsilon}}{2}]$ and applying
the  Borel-Cantelli Lemma, we see that   there exists an
$\Omega_0\subseteq\Omega$ with $\mathbb{P}(\Omega_0)=1$ such that
for any $\omega\in\Omega_0$ there exists an integer
$n_0=n_0(\omega)>0$ if $n\geq n_0$
\begin{equation*}
\begin{split}
&\int_0^t\frac{\langle
U_x(s,X(s),r(s)),g(s,X(s),r(s))dW(s)\rangle_H}{U(s,X(s),r(s))}+\int_0^t\int_{\mathbb{Z}}\ln\Lambda(s,X(s^-),r(s),u)\tilde{N}(ds,du)\\
&\leq2\theta^{-1}\ln
n+\frac{\theta^2}{2}\int_0^t\frac{\Theta U(s,X(s),r(s))}{U^2(s,X(s),r(s))}ds\\
&\quad+\frac{1}{\theta}\int_0^t\int_\mathbb{Z}\Big[\Lambda^\theta(s,X(s),r(s),u)-1-\theta
\ln \Lambda(s,X(s),r(s),u)\Big]\lambda(du)ds
\end{split}
\end{equation*}
and
\begin{equation*}
\begin{split}
\int_0^t\int_\mathbb{R}\ln
\tilde{\Upsilon}(s,\ell)\mu(ds,d\ell)\leq2\theta^{-1}\ln
n+\frac{1}{\theta}\int_0^t\int_{\mathbb{R}}\Big[\tilde{\Upsilon}^{\theta}(s,\ell)-1-\theta\ln
\tilde{\Upsilon}(s,\ell)\Big]m(d\ell)ds
\end{split}
\end{equation*}
for $0\leq t\leq n$. Hence for $\omega\in\Omega_0$, $0\leq t\leq n$
and $ n\geq n_0$
\begin{equation*}
\begin{split}
\ln U(t,X(t),r(t)) &\leq\ln U(0,x_0,r_0)+4\theta^{-1}\ln
n+\int_0^tJ(s,X(s),r(s))ds\\
&\quad+\int_0^t\frac{\mathcal
{L}U(s,X(s),r(s))}{U(s,X(s),r(s))}ds-\frac{1-\theta}{2}\int_0^t\frac{\Theta U(s,X(s),r(s))}{U^2(s,X(s),r(s))}ds\\
&\quad+\int_0^t\sum_{j=1}^m\gamma_{r(s)j}\left[\ln \Psi(s,X(s),i,j)-\Psi(s,X(s),i,j)\right]ds\\
&\quad+\frac{1}{\theta}\int_0^t\int_\mathbb{Z}\Big[\Lambda^\theta(s,X(s),r(s),u)-1-\theta
\ln \Lambda(s,X(s),r(s),u)\Big]\lambda(du)ds\\
&\quad+\frac{1}{\theta}\int_0^t\int_{\mathbb{R}}\Big[\tilde{\Upsilon}^{\theta}(s,\ell)-1-\theta\ln
\tilde{\Upsilon}(s,\ell)\Big]m(d\ell)ds.
\end{split}
\end{equation*}
This, together with (ii)-(v), yields that for $\omega\in\Omega_0$,
$0\leq t\leq n$ and $ n\geq n_0$
\begin{equation}\label{eq1}
\begin{split}
\ln U(t,X(t),r(t)) &\leq\ln U(0,x_0,r_0)+4\theta^{-1}\ln
n\\
&\quad+\int_0^t[\alpha(r(s))-\frac{1-\theta}{2}\beta(r(s))-\delta(r(s))+\rho(r(s))]ds\\
&\quad+\frac{1}{\theta}\int_0^t\int_\mathbb{Z}\Gamma(s,X(s),y,r(s),\theta)\lambda(dy)ds\\
&\quad+\frac{1}{\theta}\int_0^t\int_{\mathbb{R}}\widetilde\Theta(s,\ell,\theta)m(d\ell)ds\\
&:=I_1(t)+I_2(t)+I_3(t)+I_4(t),
\end{split}
\end{equation}
where
\begin{equation*}
\Gamma(s,x,i,u,\theta):=\Lambda^\theta(s,x,i,u)-1-\theta \ln
\Lambda(s,x,i,u)
\end{equation*}
and
\begin{equation*}
\widetilde\Theta(s,\ell,\theta):=\tilde{\Upsilon}^{\theta}(s,\ell)-1-\theta\ln
\tilde{\Upsilon}(s,\ell).
\end{equation*}
Using Taylor's series expansion, for
$\theta\in(0,\frac{\epsilon\wedge\tilde{\epsilon}}{2}]$ we see that
\begin{equation*}
\Lambda^\theta(s,x,i,u)=1+\theta \ln
\Lambda(s,x,i,u)+\frac{\theta^2}{2}(\ln\Lambda(t,x,i,u))^2\Lambda^{\xi}(t,x,i,u),
\end{equation*}
where $\xi$ lies between $0$ and $\theta$. Hence
\begin{equation*}
\begin{split}
I_3(t)&=\dfrac{\theta}{2}\int_0^t\int_\mathbb{Z}(\ln\Lambda(s,X(s),r(s),u))^2\Lambda^{\xi}(s,X(s),r(s),u)\lambda(du)ds\\
&=\dfrac{\theta}{2}\int_0^t\int_{0<\Lambda\leq1}(\ln\Lambda(s,X(s),r(s),u))^2\Lambda^{\xi}(s,X(s),r(s),u)\lambda(du)ds\\
&\quad+\dfrac{\theta}{2}\int_0^t\int_{\Lambda>1}(\ln\Lambda(s,X(s),r(s),u))^2\Lambda^{\xi}(s,X(s),r(s),u)\lambda(du)ds.
\end{split}
\end{equation*}
Noting that, for $0\leq\xi\leq\frac{\epsilon}{2}$,
$\Lambda^{\xi}\leq1$ if $0<\Lambda\leq1$,
$\Lambda^{\xi}\leq\Lambda^{\frac{\epsilon}{2}}$ if $\Lambda\geq1$,
and recalling the inequality
 \begin{equation*} \ln x\leq
\frac{4}{\epsilon}(x^{\frac{\epsilon}{4}}-1)\ \  \mbox{ for }
x\geq1,
\end{equation*}
we  obtain
\begin{equation*}
I_3(t)\leq\frac{\theta}{2}\int_0^t\int_\mathbb{Z}\Big[(\ln\Lambda(s,X(s),r(s),u))^2+\frac{16}{\epsilon^2}
\Lambda^\epsilon(s,X(s),r(s),u)\Big]\lambda(du)ds.
\end{equation*}
Similarly,
\begin{equation*}
I_4(t)\leq\frac{\theta}{2}\int_0^t\int_{\mathbb{R}}\Big[(\ln\Upsilon(s,\ell))^2+\frac{16}{\tilde{\epsilon}^2}\Upsilon^{\tilde{\epsilon}}(s,\ell)\Big]m(d\ell)ds.
\end{equation*}
Hence, by (i) for $\omega\in\Omega_0$, $n-1\leq t\leq n$ and $ n\geq
n_0+1$
\begin{equation*}
\begin{split}
\frac{1}{t}\ln(\|X(t)\|_H)&\leq-\frac{\ln\rho}{pt}+\frac{1}{pt}\Big[\ln
U(0,x_0,r_0)+4\theta^{-1}\ln
n+\frac{8\theta}{(\epsilon\wedge\tilde{\epsilon})^2} (\Xi(t)+\Pi(t))\\
&\quad+\int_0^t[\alpha(r(s))-\frac{1-\alpha}{2}\beta(r(s))-\delta(r(s))+\rho(r(s))]ds\Big].
\end{split}
\end{equation*}
Taking into account the ergodic property of  Markovian chains, e.g.,
\cite[Theorem 3.8.1, p126]{n97}, and combining  (v) with (vi), we
have almost surely
\begin{equation*}
\begin{split}
\limsup_{t\rightarrow\infty}\frac{1}{t}\ln(\|X(t)\|_H)&\leq\frac{1}{p}\left[\frac{8\theta}{(\epsilon\wedge\tilde\epsilon)^2}
(\zeta+\eta)+\sum\limits_{i=1}^m\pi_i\left(\alpha_i-\frac{1-\theta }{2}\beta_i-\delta_i+\rho_i\right)\right],\\
\end{split}
\end{equation*}
and the conclusion follows by the arbitrariness of  $\theta$.

Now we give an example to demonstrate Theorem \ref{asymptotic}.

\begin{example}
{\rm Let $r(t)$ be a right-continuous Markov chain taking values in
$\mathbb{S}=\{1,2\}$ with the generator
$\Gamma=(q_{ij})_{2\times2}$:
\begin{equation*}
-q_{11}=q_{12}=1,\ \ \ \ \ -q_{22}=q_{21}=q>0.
\end{equation*}
The unique stationary probability distribution of the Markov chain
$r(t)$ is
\begin{equation*}
\pi=(\pi_1,\pi_2)=(q/(1+q),1/(1+q)).
\end{equation*}
Let $f,g:\mathbb{R}\times\mathbb{S}\rightarrow\mathbb{R}$ be
Lipschitz continuous in the first argument and satisfy linear growth
conditions. Assume that  there exist constants $b_i\in\mathbb{R},
d_i>0, i=1,2$, such that for $x\in\mathbb{R}$
\begin{equation}\label{eq11}
2xf(x,i)+g^2(x,i)\leq b_ix^2,
\end{equation}
and
\begin{equation}\label{eq12}
xg(x,i)\geq d_i^{\frac{1}{2}}x^2.
\end{equation}
For  $i=1,2$ let
\begin{equation*}
\begin{split}
&\delta_i:=\int_0^\infty[\gamma^2_i(y)+2\gamma_i(y)-2\ln(1+\gamma_i(y))]\lambda(dy)>0,\\
&m_i:=\int_0^\infty[2\gamma_i(y)-2\ln(1+\gamma_i(y))]\lambda(dy).
\end{split}
\end{equation*}
Assume further that
\begin{equation}\label{eq14}
\int_0^\infty[(\ln(1+\gamma_i(y)))^2+\gamma_i^2(y)]\lambda(dy)<\infty,
\ \ \ \ i=1,2.
\end{equation}

Consider the following equation
\begin{equation}\label{eq29}
\begin{split}
&dX(t)=\left[AX(t)+f(X(t),r(t))\right]dt+g(X(t),r(t))dW(t)\\
&\quad\ \ \ \ \ \ \ \ \  +\int_0^\infty\gamma(r(t),y)X(t^-)\tilde{N}(dy,dt), \ \ t>0, x\in(0,\pi);\\
&X(0,x)=u_0(x), \ \ x\in(0,\pi);\ \ \ X(t,0)=X(t,\pi)=0, \ \ t\geq0.
\end{split}
\end{equation}
In this example, set $H:=L^2([0,\pi]), A=\frac{\partial }{\partial
x}(a(x)\frac{\partial}{\partial x})$ with domain $\mathcal {D}(A)$
satisfying the boundary conditions above. We let $a(x)$ be a
measurable function defined on $[0,\pi]$ such that
\begin{equation}\label{eq18}
0<\nu\leq a(x),
\end{equation}
for some positive constant $\nu$.

Let $U(t,u,i):=\lambda_i\|u\|_H^2, u\in H, i=1,2$, where
$\lambda_1=1$ and $\lambda_2$ is a positive constant which will be
determined later. Note from \eqref{eq11}, \eqref{eq18} and
Poincar\'e's inequality that for $u\in \mathcal {D}(A)$
\begin{equation*}
\begin{split}
\mathcal {L}U(t,u,1)&=2\langle Au+f(u,1),
u\rangle_H+\|g(u,1)\|_H^2\\
&\quad+\int_0^\infty\gamma_1^2(y)\lambda(dy)\|u\|_H^2+q_{11}\lambda_1\|u\|_H^2+q_{12}\lambda_2\|u\|_H^2\\
&\leq\left[-2\nu+b_1+\int_0^\infty\gamma_1^2(y)\lambda(dy)+\lambda_2-1\right]U(t,u,1)\\
&:=\alpha_1 U(t,u,1),
\end{split}
\end{equation*}
and similarly
\begin{equation*}
\begin{split}
\mathcal {L}U(t,u,2)&\leq\Big[-2\nu+b_2+\int_0^\infty\gamma_2^2(y)\lambda(dy)+q\Big(\frac{1}{\lambda_2}-1\Big)\Big]U(t,u,2)\\
&:=\alpha_2 U(t,u,2).
\end{split}
\end{equation*}
By the definition of $U$, it is easy to see that
\begin{equation*}
\rho_1=1-\lambda_2 +\ln\lambda_2\mbox{ and }
\rho_2=q\Big(1-\frac{1}{\lambda_2}-\ln\lambda_2\Big).
\end{equation*}
From \eqref{eq12}, it follows that
\begin{equation*}
\Theta U(t,u,i)\geq
d_i\|u\|_H^4=\frac{d_i}{\lambda_i^2}U^2(t,u,i):=\beta_i U^2(t,u,i).
\end{equation*}
Moreover, (v) follows from \eqref{eq14} and (iv), and (vi) holds due
to the definition of $U$. Thus, by Theorem \ref{asymptotic} we
arrive at
\begin{equation*}
\begin{split}
\limsup_{t\rightarrow\infty}\frac{1}{t}\ln(\|X(t)\|_H)&\leq-\frac{\vartheta}{2(1+q)},
\ \ \mbox{ a.s., }
\end{split}
\end{equation*}
where
\begin{equation*}
\vartheta:=q\Big(\frac{d_1}{2}+m_1+2\nu-b_1\Big)
+\frac{d_2}{2\lambda_2^2}+m_2+2\nu-b_2.
\end{equation*}
In particular, let
\begin{equation*}
\frac{d_1}{2}+m_1+2\nu-b_1<0,
\end{equation*}
and choose $\lambda_2>0$ such that
\begin{equation*}
\frac{d_2}{2\lambda_2^2}+m_2+2\nu-b_2>0.
\end{equation*}
Then Eq. \eqref{eq29} is almost surely exponentially stable whenever
\begin{equation*}
0<q<-\Big(\frac{d_2}{2\lambda_2^2}+m_2+2\nu-b_2\Big)\Big/\Big(\frac{d_1}{2}+m_1+2\nu-b_1\Big).
\end{equation*}
}
\end{example}

\section{Linear Switching-diffusion SPDEs with Jumps}
In this section to demonstrate that the results obtained in Theorem
\ref{asymptotic} are sharp,  we will discuss a class of linear
switching-diffusion SPDEs with jumps.

For a bounded domain $\mathcal {O}\subset\mathbb{R}^n$ with
$C^{\infty}$ boundary $\partial\mathcal {O}$, let $H:=L^2(\mathcal
{O})$ denote the family of all real-valued square integrable
functions, equipped with the usual inner product $\langle
f,g\rangle_H:=\int_{\mathcal {O}}f(x)g(x)dx, f,g\in H$ and norm
$\|f\|_H:=\left(\int_{\mathcal {O}}f^2(x)dx\right)^{\frac{1}{2}},
f\in H$. Let $A:=\sum_{i=1}^n\frac{\partial^2}{\partial x^2}$ be the
classical Laplace operator  defined on the Sobolev space  $\mathcal
{D}(A):= H^1_0(\mathcal {O})\cap H^2(\mathcal {O})$, where
$H^m(\mathcal {O}),m=1,2$, consist of functions of $H$ whose
derivatives $D^{\alpha}u$, in the sense of distributions, of order
$|\alpha|\leq m$ are in $H$. It is well known that there exists an
orthonormal basis of $H$, $\{e_n\}_{n\geq1}, n=1,2,\ldots$,
satisfying
\begin{equation}\label{eq00}
e_n\in \mathcal {D}(A),\ \ \ \ -Ae_n=\lambda_ne_n.
\end{equation}
Thus, for any $f\in H$, we can write
\begin{equation*}
f=\sum\limits_{n=1}^{\infty}f_ne_n, \mbox{ where } f_n=\langle f,
e_n\rangle.
\end{equation*}

Consider the following SPDE driven by a switching-diffusion process
with jumps
\begin{equation}\label{eq2}
\begin{cases}
\ \ dX(t)=(A X(t)+\bar{\alpha}(r(t)) X(t))dt+\bar{\beta}(r(t))
X(t)dW(t)\\
\ \ \ \ \ \ \ \ \ \
\quad+\int_0^\infty\bar{\gamma}(r(t),y)X(t^-)\tilde{N}(dt,dy),
  x\in\mathcal {O}, \ t > 0,\\
\ X(t,x)=0,  x\in\partial\mathcal {O},\ t>0,\\
\ X(0,x)=u^0(x),  x\in\mathcal {O} \mbox{ and } r(0)=r_0.
\end{cases}
\end{equation}
Here $\bar{\alpha},\bar{\beta}:\mathbb{S}\rightarrow\mathbb{R},
\bar{\gamma}:\mathbb{S}\times[0,\infty)\rightarrow\mathbb{R}$, $W$
is a real-valued Wiener process on the probability space
$\{\Omega,{\mathcal F}, \mathbb{P}\}$.

Applying Theorem \ref{asymptotic} we can deduce the following
conclusion.

\begin{theorem}\label{linear}
{\rm Assume that for $i\in \mathbb{S}, y\in(0,\infty)$
\begin{equation}\label{eq5}
\bar{\gamma}_i(y):=\bar{\gamma}(i,y)>-1,
  \ \ \ \max_{1\leq i\leq
m}\int_0^\infty\bar{\gamma}^2_i(y)\lambda(dy)<\infty,
\end{equation}
and
\begin{equation}\label{eq6}
\max_{1\leq i\leq
m}\int_0^\infty(\ln(1+\bar{\gamma}_i(y)))^2\lambda(dy)<\infty.
\end{equation}
Then  Eq. \eqref{eq2} has the property
\begin{equation*}
\limsup\limits_{t\rightarrow\infty}\frac{1}{t}\ln(\|X(t)\|_H)\leq-\lambda_1+\sum\limits_{j=1}^{m}\pi_j\left(\bar{\alpha}_j-\frac{1}{2}\bar{\beta}_j^2+\int_0^\infty[\ln(1+\bar{\gamma}_i(y))-\bar{\gamma}_i(y)]\lambda(dy)\right)
\quad a.s.
\end{equation*}
In particular, the solution  of Eq. \eqref{eq2} is almost surely
exponentially stable if
\begin{equation*}
\lambda_1>\sum\limits_{j=1}^{m}\pi_j\left(\bar{\alpha}_j-\frac{1}{2}\bar{\beta}_j^2+\int_0^\infty[\ln(1+\bar{\gamma}_i(y))-\bar{\gamma}_i(y)]\lambda(dy)\right).
\end{equation*}
}
\end{theorem}

\begin{proof}
Under  \eqref{eq5}, by Lemma \ref{global existence} Eq. \eqref{eq2}
admits a unique global mild solution. Let $U(u):=\|u\|_H^2, u\in H$,
and set for $t\geq0, u\in H, i\in\mathbb{S}$ and $y\in\mathbb{Z}$
\begin{equation*}
A(t,u,i):=Au+\bar{\alpha}_iu,\ \  g(t,u,i):=\bar{\beta}_iu\ \ \mbox{
and }\ \ \Phi(t,u,i,y):=\bar{\gamma}_i(y)u.
\end{equation*}
It is easy to compute $\alpha_i,\beta_i$ and $\delta_i$ in Theorem
\ref{asymptotic} as follows
\begin{equation*}
\alpha_i=-2\lambda_1+2\bar{\alpha}_i+\bar{\beta}^2_i+\int_0^\infty\bar{\gamma}_i^2(y)\lambda(dy),\
\ \  \beta_i=4\bar{\beta}_i^2,
\end{equation*}
and
\begin{equation*}
\delta_i=\int_0^\infty[2(\bar{\gamma}_i(y)-\ln(1+\bar{\gamma}_i(y)))+\bar{\gamma}_i^2(y)]\lambda(dy).
\end{equation*}
Moreover, noting that
\begin{equation*}
\Lambda(t,u,i,y)=(1+\bar{\gamma}_i(y))^2 \mbox{ and } \Upsilon(t, u,
r_0, i,\ell)=1,
\end{equation*}
together with  \eqref{eq5} and \eqref{eq6}, we can deduce that  (iv)
and (vi) hold. Then the desired assertion follows by Theorem
\ref{asymptotic}.
\end{proof}

We now discuss  the sample path stability of the solution to Eq.
\eqref{eq2} using its explicit mild solution, which will be
 given in Lemma \ref{explicit solution} below. In the sequel, when $u^0$ is
deterministic and $u^0\not= 0$, we set $n_0:=\inf\{n:u^0_n\neq 0\}$,
where $u_n^0:=\langle u^0,e_n\rangle$ for $n\geq1$.
\begin{lemma}\label{explicit solution}
{\rm Under  \eqref{eq5},  the unique global mild solution of Eq.
\eqref{eq2} has the explicit form
\begin{equation}\label{eq3}
\begin{split}
X(t,x)&=v(t,x)\exp\Big\{-\frac{1}{2}\int_0^t\bar{\beta}^2(r(s))ds+\int_0^t\int_0^\infty[\ln(1+\bar{\gamma}(r(s),y))-\bar{\gamma}(r(s),y)]\lambda(dy)ds\\
&\quad+\int_0^t\bar{\beta}(r(s))dW(s)+\int_0^t\int_0^\infty\ln(1+\bar{\gamma}(r(s),y))\tilde{N}(ds,dy)\Big\},
\end{split}
\end{equation}
where
\begin{equation}\label{eq4}
v(t,x):=\sum\limits_{n=1}^{\infty}\exp\left\{-\lambda_nt+\int_0^t\bar{\alpha}(r(s))ds\right\}u_n^0
e_n(x), \ \ \ t\geq0, \ \ x\in\mathcal {O}.
\end{equation}
}
\end{lemma}

\begin{proof}
Under  \eqref{eq5}, Eq. \eqref{eq2} has a unique global mild
solution such that
\begin{equation}\label{eq34}
\begin{split}
X(t)&=S(t)u^0+\int_0^t\bar{\alpha}(r(s))S(t-s)
X(s))ds+\int_0^t\bar{\beta}(r(s))S(t-s) X(s))dW(s)\\
&\quad+\int_0^t\int_0^\infty\bar{\gamma}(r(s),y)S(t-s)X(s^-)\tilde{N}(ds,dy).
\end{split}
\end{equation}
 Since
\begin{equation*}
S(t)u=\sum_{n=1}^\infty e^{-\lambda_nt}\langle u,e_n\rangle_He_n
\mbox{ for }u\in H,
\end{equation*}
Eq. \eqref{eq34} can be rewritten in the form
\begin{equation*}
\begin{split}
X(t)&=\sum_{n=1}^\infty e^{-\lambda_nt}\langle u^0,e_n\rangle_He_n+\int_0^t\bar{\alpha}(r(s))\sum_{n=1}^\infty e^{-\lambda_n(t-s)}\langle X(s),e_n\rangle_He_nds\\
&\quad+\int_0^t\bar{\beta}(r(s))\sum_{n=1}^\infty e^{-\lambda_n(t-s)}\langle X(s),e_n\rangle_He_ndW(s)\\
&\quad+\int_0^t\int_0^\infty\bar{\gamma}(r(s),y)\sum_{n=1}^\infty
e^{-\lambda_n(t-s)}\langle X(s^-),e_n\rangle_He_n\tilde{N}(ds,dy).
\end{split}
\end{equation*}
This  further yields
\begin{equation*}
\begin{split}
e^{\lambda_nt}\langle X(t),e_n\rangle_H&=\langle u^0,e_n\rangle_H+\int_0^t\bar{\alpha}(r(s)) e^{\lambda_ns}\langle X(s),e_n\rangle_Hds\\
&\quad+\int_0^t\bar{\beta}(r(s)) e^{\lambda_ns}\langle X(s),e_n\rangle_HdW(s)\\
&\quad+\int_0^t\int_0^\infty\bar{\gamma}(r(s),y)
e^{\lambda_ns}\langle X(s^-),e_n\rangle_H\tilde{N}(ds,dy).
\end{split}
\end{equation*}
Then It\^o's formula gives
\begin{equation*}
\begin{split}
\langle X(t),e_n\rangle_H&=\langle
u^0,e_n\rangle_H\exp\Big\{-\lambda_nt+\int_0^t\Big[\bar{\alpha}(r(s))-\frac{1}{2}\bar{\beta}^2(r(s))\Big]ds\\
&\quad+\int_0^t\int_0^\infty[\ln(1+\bar{\gamma}(r(s),y))-\bar{\gamma}(r(s),y)]\lambda(dy)ds\\
&\quad+\int_0^t\bar{\beta}(r(s))dW(s)+\int_0^t\int_0^\infty\ln(1+\bar{\gamma}(r(s),y))\tilde{N}(ds,dy)\Big\}.
\end{split}
\end{equation*}
The desired assertion therefore follows by observing that
\begin{equation*}
X(t)=\sum_{n=1}^\infty \langle X(t),e_n\rangle_He_n.
\end{equation*}
\end{proof}

\begin{theorem}\label{direct method}
{\rm Under the conditions of Theorem \ref{linear}, Eq. \eqref{eq2} has the following properties:\\

\begin{enumerate}
\item[\textmd{(i)}] If $u^0$ is deterministic and $u^0\not= 0$, then
\begin{equation*}
\begin{split}
\lim\limits_{t\rightarrow\infty}\frac{1}{t}\ln(\|X(t)\|_H)&=-\lambda_{n_0}+\sum\limits_{j=1}^{m}\pi_j\bigg(\bar{\alpha}_j-\frac{1}{2}\bar{\beta}_j^2\\
&\quad\ \
+\int_0^\infty[\ln(1+\bar{\gamma}_i(y))-\bar{\gamma}_i(y)]\lambda(dy)\bigg)
\quad a.s.
\end{split}
\end{equation*}
In particular, the solution of Eq. \eqref{eq2} with initial
condition $u^0$ will converge exponentially to zero  with
probability one if and only if
\begin{equation*}
\lambda_{n_0}>\sum\limits_{j=1}^{m}\pi_j\left(\bar{\alpha}_j-\frac{1}{2}\bar{\beta}_j^2+\int_0^\infty[\ln(1+\bar{\gamma}_i(y))-\bar{\gamma}_i(y)]\lambda(dy)\right).
\end{equation*}
\item[\textmd{(ii)}] For any initial condition $u^0$,
\begin{equation*}
\begin{split}
\limsup\limits_{t\rightarrow\infty}\frac{1}{t}\ln(\|X(t)\|_H)&\leq-\lambda_1+\sum\limits_{j=1}^m\pi_j\bigg(\bar{\alpha}_j-\frac{1}{2}\bar{\beta}_j^2\\
&\quad\
+\int_0^\infty[\ln(1+\bar{\gamma}_i(y))-\bar{\gamma}_i(y)]\lambda(dy)\bigg)
\quad a.s.
\end{split}
\end{equation*}
In particular, the solution  of Eq. \eqref{eq2} is almost surely
exponentially stable if
\begin{equation*}
\lambda_1>\sum\limits_{j=1}^{m}\pi_j\left(\bar{\alpha}_j-\frac{1}{2}\bar{\beta}_j^2+\int_0^\infty[\ln(1+\bar{\gamma}_i(y))-\bar{\gamma}_i(y)]\lambda(dy)\right).
\end{equation*}
\end{enumerate}}
\end{theorem}

\begin{proof}
Note that
\begin{equation*}
\begin{split}
\frac{1}{t}\ln(\|X(t)\|_H)&=\frac{1}{t}\ln(\|v(t)\|_H)-\frac{1}{2t}\int_0^t\bar{\beta}^2(r(s))ds\\
&\quad+\frac{1}{t}\int_0^t\int_0^\infty[\ln(1+\bar{\gamma}(r(s),y))-\bar{\gamma}(r(s),y)]\lambda(dy)ds\\
&\quad+\frac{1}{t}M_1(t)+\frac{1}{t}M_2(t),
\end{split}
\end{equation*}
where
\begin{equation*}
M_1(t):=\int_0^t\bar{\beta}(r(s))dW(s) \mbox{ and }
M_2(t):=\int_0^t\int_0^\infty\ln(1+\bar{\gamma}(r(s),y))\tilde{N}(ds,dy).
\end{equation*}
Since
\begin{equation*}
\langle M_1\rangle_t\leq t\max_{1\leq i\leq m}\beta_i^2 \mbox{ and
}\langle M_2\rangle_t\leq t\max_{1\leq i\leq
m}\int_0^\infty(\ln(1+\bar{\gamma}_i(y)))^2\lambda(dy),
\end{equation*}
together with the strong law of large numbers for local martingales,
e.g., Lipster \cite{l80}, we  have
\begin{equation*}
\lim\limits_{t\rightarrow\infty}\frac{1}{t}M_1(t)=\lim\limits_{t\rightarrow\infty}\frac{1}{t}M_2(t)=0\
\ \mbox{ a.s. }
\end{equation*}
Furthermore,  by the ergodic property of Markov chains, e.g.,
\cite[Theorem 3.8.1, p126]{n97}
\begin{equation*}
\begin{split}
&-\frac{1}{2t}\int_0^t\bar{\beta}^2(r(s))ds
+\frac{1}{t}\int_0^t\int_0^\infty[\ln(1+\bar{\gamma}(r(s),y))-\bar{\gamma}(r(s),y)]\lambda(dy)ds\\
&\rightarrow\sum\limits_{i=1}^m\pi_i\left(-\frac{1}{2}\bar{\beta}_i^2+\int_0^\infty[\ln(1+\bar{\gamma}_i(y))-\bar{\gamma}_i(y)]\lambda(dy)\right)
\ \ \mbox{ a.s. }
\end{split}
\end{equation*}
whenever  $t\rightarrow\infty$. Moreover, it is not difficult to
show that
\begin{equation*}
\lim_{t \rightarrow \infty}\frac{1}{t}\ln(\|v(t)\|_H)
\begin{cases}
&=-\lambda_{n_0}+\sum_{j=1}^m\pi_j \bar \alpha_j, \mbox{ if } u^{0}\neq 0 \mbox{ is deterministic },\\
&\le \lambda_{1}+\sum_{j=1}^m\pi_j \bar \alpha_j, \mbox{ for any
initial condition  } u^{0}.
\end{cases}
\end{equation*}
The proof is therefore complete.
\end{proof}

\begin{example}\label{example}
{\rm Let $r(t)$ be a right-continuous Markov chain taking values in
$\mathbb{S}=\{1,2\}$ with the generator
$\Gamma=(q_{ij})_{2\times2}$:
\begin{equation*}
-q_{11}=q_{12}=\nu>0,\ \ \ \ \ -q_{22}=q_{21}=q>0.
\end{equation*}
Then the unique stationary probability distribution of the Markov
chain $r(t)$ is
\begin{equation*}
\pi=(\pi_1,\pi_2)=(q/(\nu+q),\nu/(\nu+q)).
\end{equation*}
For $i=1,2$ set
\begin{equation*}
\mu_i:=\int_0^\infty[\ln(1+\bar{\gamma}_i(y))-\bar{\gamma}_i(y)]\lambda(dy).
\end{equation*}
Let $\bar{\alpha}_1,\bar{\alpha}_2\in\mathbb{R}$ such that
\begin{equation}\label{eq9}
\bar{\alpha}_1+\mu_1>1, \ \ \bar{\alpha}_2+\mu_2<1,
\end{equation}
and choose $q$ obeying
\begin{equation}\label{eq10}
0<q<\nu(1-\bar{\alpha}_2-\mu_2)/(\bar{\alpha}_1+\mu_1-1).
\end{equation}
Consider the switching-diffusion equation with jumps:
\begin{equation}\label{eq8}
\begin{cases}
\ \ dX(t)=(A X(t)+\bar{\alpha}(r(t)) X(t))dt\\
\ \ \ \ \ \ \ \ \
\quad+\int_0^\infty\bar{\gamma}(r(t),y)X(t^-)\tilde{N}(dt,dy),
  t>0, x\in(0,\pi),\\
\ X(t,0)=X(t,\pi)=0,   \ \ t>0,\\
X(0,x)=u^0(x)=\sqrt{2/\pi}\sin x,\ \ \ \ x\in(0,\pi).
\end{cases}
\end{equation}
The previous stochastic system  can be regarded as the combination
of two equations
\begin{equation}\label{eq16}
dX(t)=(A X(t)+\bar{\alpha}_1
X(t))dt+\int_0^\infty\bar{\gamma}_1(y)X(t^-)\tilde{N}(dt,dy),\
t\geq0, x\in(0,\pi),
\end{equation}
and
\begin{equation}\label{eq0}
dX(t)=(A X(t)+\bar{\alpha}_2
X(t))dt+\int_0^\infty\bar{\gamma}_2(y)X(t^-)\tilde{N}(dt,dy),\
t\geq0, x\in(0,\pi)
\end{equation}
with the same Dirichlet boundary condition and initial condition as
Eq. \eqref{eq8}, switching from one to the other according to the
law of the Markov chain. Recalling that $e_n(x)=\sqrt{2/\pi}\sin nx,
n=1,2,3,\ldots,$ are the eigenfunctions of $-A$, with positive and
increasing eigenvalues $\lambda_n=n^2$, we hence have $\lambda_1=1$.
By Theorem \ref{linear}, we have
\begin{equation*}
\limsup\limits_{t\rightarrow\infty}\frac{1}{t}\ln(\|X(t)\|_H)\leq-1+\pi_1\bar{\alpha}_1+\pi_2\bar{\alpha}_2+\pi_1\mu_1+\pi_2\mu_2
\quad a.s.
\end{equation*}
This, together with \eqref{eq9} and \eqref{eq10}, yields that Eq.
\eqref{eq8} is almost surely exponentially stable. On the other
hand,
 note that the initial condition $u^0(x)=\sqrt{2/\pi}\sin x$ is deterministic and
$u_1^0=1$, which implies $n_0=1$. By Theorem \ref{direct method},
the solution to Eq.\eqref{eq16} has the property
\begin{equation*}
\lim\limits_{t\rightarrow\infty}\frac{1}{t}\ln(\|X(t)\|_H)=-1+\bar{\alpha}_1+\mu_1,
  \ \ \ \mbox{ a.s.,}
\end{equation*}
and the solution to Eq.\eqref{eq0}  has the property
\begin{equation*}
\lim\limits_{t\rightarrow\infty}\frac{1}{t}\ln(\|X(t)\|_H)=-1+\bar{\alpha}_2+\mu_2,\
\ \  \mbox{ a.s.}
\end{equation*}
Then, by \eqref{eq9} the solution of  stochastic system \eqref{eq16}
explodes exponentially, and the solution of stochastic system
\eqref{eq0} converges exponentially to zero. However, we observe
that due to the Markovian switching the overall system \eqref{eq8}
is almost surely exponentially stable. }
\end{example}

\noindent{\bf{Acknowledge }} The authors wish to express their
sincere thanks to the anonymous referee for his/her careful comments
and valuable suggestions, which greatly improved the paper. The
authors also would like to thank Dr Andrew Neate for his comments.

\bibliographystyle{amsplain}

\end{document}